\documentclass{article}
\usepackage [utf8] {inputenc}
\usepackage [english] {babel}
\usepackage{mathtools}
\usepackage{amsthm}
\usepackage{amssymb}
\usepackage{enumerate}

\newtheorem*{theorem}{\textit{Theorem}}
\newtheorem{lem}{\textit{Lemma}}
\newtheorem*{cor}{\textit{Corollary}}

\DeclareMathOperator{\Spin}{Spin}
\DeclareMathOperator{\HS}{HS}

\title{On L.\,G.\,Kov\`acs' problem}
\author{A.\,V.\,Vasil'ev and  S.\,V.\,Skresanov}

\begin{document}

\date{\vspace{-5ex}}
{\let\newpage\relax\maketitle}

\begin{abstract}
``Kourovka notebook'' contains the question due to L.\,G.\,Kov\`acs (Problem~8.23): If the dihedral group $D$ of order 18 is a section of a direct product $X\times Y$, must at least one of $X$ and $Y$ have a section isomorphic to~$D$? The goal of our short paper is to give the positive answer to this question provided that $X$ and $Y$ are locally finite. In fact, we prove even more: If a non-trivial semidirect product $D=A\rtimes B$ of a cyclic $p$-group $A$ and a group $B$ of order $q$, where $p$ and $q$ are distinct primes, lies in a locally finite variety generated by a set $\mathfrak{X}$ of groups, then $D$ is a section of a group from~$\mathfrak{X}$.
\end{abstract}

~

``Kourovka notebook'' contains the following question due to L.\,G.\,Kov\`acs \cite[Problem~8.23]{kourovka}: If the dihedral group $D$ of order 18 is a section of a direct product $X\times Y$, must at least one of $X$ and $Y$ have a section isomorphic to~$D$? Recall that a section of a group is a homomorphic image of one of its subgroups. If $X$ and $Y$ are locally finite groups, the positive answer to the question is given by the following

\begin{theorem}
Let $p$ and $q$ be distinct primes and let $D=A\rtimes B$ be a semidirect product of its subgroups $A$ and $B$ such that  $A$ is a cyclic $p$-group, $B$ is a cyclic group of order $q$, and $[A, B]\neq1$. Suppose that $X$ and $Y$ are locally finite groups and $D$ is a section of the direct product $X\times Y$. Then $D$ is a section of either $X$ or~$Y$.
\end{theorem}

As we will see below, it suffices to prove the theorem in the case of finite groups $X$ and $Y$. Note that the theorem for $q=2$ (which means that $D$ is a dihedral group) and finite groups $X$ and $Y$ was proved by the first author earlier, but the paper \cite{vasand}, where it was done, is virtually unavailable. In fact, the part of the proof from \cite{vasand} which does not use the specific structure of the group $D$ is reproduced here almost without any change.

We recall that a variety of groups is locally finite if every finitely generated group from this variety is finite. The direct consequence of our theorem and old G.\,Higman's result \cite[Lemma 4.3]{Higman} is the following

\begin{cor} Let $\mathfrak{L}$ be a locally finite variety generated by a set $\mathfrak{X}$ of groups, and let $D$ be as in the theorem. If $D$ lies in $\mathfrak{L}$, then $D$ is a section of a group from $\mathfrak{X}$. In particular, if $\mathfrak{X}$ consists of a finite number of finite groups and $D$ belongs to the variety generated by $\mathfrak{X}$, then $D$ is a section of a group from~$\mathfrak{X}$.
\end{cor}

It is worth mentioning that such a group $D$ turns out to be critical, i.\,e., it does not lie in the variety generated by its proper sections
(the definition and properties of critical groups can be found, for example, in \cite{critical}).

\medskip

\textbf{Proof of the theorem.} We start with some properties of the group~$D$. Suppose that the order of its subgroup $A$ is $p^n$, and $A=\langle a\rangle$.

\begin{lem}\label{PropOfD}
Proper normal subgroups of $D$ are $p$-groups. In particular, the lattice of normal subgroups of $D$ is a chain.
\end{lem}

\begin{proof}
As $B$ acts on $A$ by $p'$-automorphisms, it follows from \cite[Theorem 5.2.3]{gorenstein} that $A$ is a direct product of $[A,B]$ and~$C_A(B)$. Since $A$ is cyclic, it is equal to either $[A,B]$ or $C_A(B)$, but the latter contradicts the hypothesis of the theorem. Thus $C_A(B)=1$.
		
Consider now an arbitrary normal subgroup $N$ of~$D$. Suppose that the order of $N$ is a multiple of~$q$. Then without loss of generality we may assume that $B\leq N$. Since $N$ is normal, we have $B^D \leq N$. On the other hand, $N_G(B)=N_A(B)B=C_A(B)B=B$. It follows that $$|N|\geq|B^D|=|\bigcup_{i=0}^{p^n-1}B^{a^i}|=1+\sum_{i=0}^{p^n-1}(|B^{a^i}|-1)=1+p^n(q-1)>\frac{p^nq}{2}=\frac{|D|}{2},$$
so $N$ coincides with~$D$. Thus each proper normal subgroup of $D$ is a $p$-group. In particular, normal subgroups of $D$ form a chain.
\end{proof}

Set $F=X\times Y$. Let $G$ be a subgroup of $F$, and let $H$ be a normal subgroup of $G$ such that $G/H\simeq D$. Put $S_X=\{x \in X \mid \exists y\in Y : (x,y) \in S\}$ for the projection of a subgroup $S$ of $G$ on $X$ and define the projection $S_Y$ of $S$ on $Y$ in the same way. In particular, $G_X$ and $G_Y$ are projections of $G$ on $X$ and $Y$.

\begin{lem}\label{PropOfXY}
We may assume that $X=G_X$ and $Y=G_Y$ and both of these groups are finite.
\end{lem}

\begin{proof} If $G_X$ or $G_Y$ contains a section isomorphic to $D$, so does $X$ or $Y$ respectively. It follows that we may assume that $X=G_X$ and $Y=G_Y$. Now it suffices to prove that we can restrict ourselves to the case of a finite group~$G$. Taking elements $u$ and $v$ of $G$ such that $Hu$ and $Hv$ generate $G/H$, define $\widetilde{G}=\langle u,v\rangle$. Then $G=\widetilde{G}H$, so $D\simeq G/H\simeq \widetilde{G}/(\widetilde{G}\cap H)$. Hence we can replace $G$ by $\widetilde{G}$. Therefore we may assume that $G$ is finitely generated. Since $G$ is a subgroup of the locally finite group~$F$, it must be finite. \end{proof}

Let $K_1=G\cap G_X$ and $K_2=G\cap G_Y$ stand for the kernels of homomorphisms $G\to G_Y$ and $G\to G_X$ respectively.

\begin{lem}\label{PropOfKi}
We may assume that $K_1\cap H=K_2\cap H=1$ and $K_1, K_2$ are $p$-groups.
\end{lem}
\begin{proof}
Note that if $S$ is a normal subgroup of $G$, then its projections $S_X$ and $S_Y$ are normal in~$F$. Indeed, let $(x,y)\in F$ and $(x_0,1)\in S_X $. Then $(x_0,1)^{(x,y)} = (x_0^x,1)$. Since $X=G_X$, there exist $y_0, y_1\in Y $ such that $(x_0,y_0)\in S $ and $(x,y_1)\in G$. Hence $(x_0,y_0)^{(x,y_1)} = (x_0^x,y_0^{y_1})\in S$, so $ (x_0^x,1)\in S_X$. The proof is similar for $S_Y$. Since $H$ is normal in $G$, the groups $H_X$ and $H_Y$ are normal in~$F$.

The groups $L_X=H\cap H_X$ and $L_Y=H\cap H_Y$ are normal in $G$ and coincide with their projections on $X$ and $Y$ respectively, so they are normal in $F$ as well. Consider now the quotient $F/(L_X\times L_Y)$. We have
$$F/(L_X\times L_Y)\simeq G_X/L_X \times G_Y/L_Y \mbox{ and } (G/(L_X \times L_Y))/(H/(L_X \times L_Y))\simeq G/H.$$
So we may assume that $L_X=L_Y =1$ and, consequently, $K_1\cap H=K_2\cap H=1$.

Thus for $i=1,2$, we obtain that $K_i\simeq K_i/(K_i\cap H)\simeq K_iH/H$ and $K_iH$ is normal in $G$, so $K_i$ is isomorphic to a normal subgroup of~$D$. If one of $K_i$ is isomorphic to $D$ then the theorem is proved, so we may suppose that $K_1$ and $K_2$ are isomorphic to proper normal subgroups of~$D$. Lemma~\ref{PropOfD} yields that $K_1$ and $K_2$ are $p$-groups.
\end{proof}

\begin{lem}\label{PropOfH}
We may assume that $H$ is a $p$-group.
\end{lem}
\begin{proof}
Let $r$ be a prime and $S$ a Sylow $r$-subgroup of $H$. By the Frattini argument, $G=HN_G(S)$. It follows that
$$D\simeq G/H = HN_G(S)/H \simeq N_G(S)/(H \cap N_G(S)),$$
so we can replace $G$ by~$N_G(S)$. Therefore, we may assume that Sylow $r$-subgroups of $H$ are normal in $G$ for all primes~$r$. Let now $S$ be a Sylow $r$-subgroup of $H$ with $r\neq p$. Then $K_1S\cap K_2S=S$ and the Remak theorem implies that $G/S$ is a subdirect product of $G/K_1S$ and $G/K_2S$. Note that $D\simeq G/H \simeq (G/S)/(H/S)$ is a section of $G/S$, as well as $G/K_1S\simeq(G/K_1)/(K_1S/K_1)=G_Y/(K_1S/K_1)$ is a section of $Y$ and $G/K_2S$ is a section of~$X$. So it suffices to find a section isomorphic to $D$ in one of $G/K_iS$, $i=1,2,$ and we may suppose that $H$ is an $r'$-group. Repeating the argument for other prime divisors of $|H|$ distinct from $p$ we derive the desired conclusion.
\end{proof}

Let $P$ stand for a Sylow $p$-subgroup and $Q$ for a Sylow $q$-subgroup of~$G$. Clearly $P$ is normal in $G$ and $Q$ is isomorphic to~$B$. Observe that $PQ=G$.

\begin{lem}\label{PropOfT}
There exists a cyclic subgroup $T$ of $P$ such that $HT=P$ and $Q$ normalizes~$T$.
\end{lem}
\begin{proof}
Let $\mathcal{G}=\{x\in P \mid H\langle x\rangle=P\} $ be the set of all preimages in $P$ of generators of $P/H\simeq A$.
Then $\mathcal{G}=\bigcup_{i=1}^m Hx_i$, where $\{Hx_i \mid i=1,\ldots,m\}$ is the set of all generators of~$P/H$. As $P/H$ is isomorphic to the cyclic subgroup of order $p^n$ and $H$ is a $p$-group, say $|H|=p^k$, we have
\begin{equation}\label{eq1}
|\mathcal{G}|=|H|m=p^k\varphi(p^n)= p^{n+k-1}(p-1),
\end{equation} where $\varphi$ is the Euler totient function.

Put $\Omega=\{\langle x\rangle \mid x\in\mathcal{G}\}$. If $T\in\Omega$ then $p^n\leq |T|\leq p^{n+k}$. For that reason the set $\Omega$ breaks into the disjoint union of the sets
$\Omega_i=\{T\in\Omega \mid |T| = p^{n+i}\}$ with $i=0,\ldots,k$. It follows right from the definition that generators of groups from $\Omega$ lie in $\mathcal{G}$ and, the other way around, each element of $\mathcal{G}$ generates a group from $\Omega$. Taking into account that two distinct groups from $\Omega$ cannot have the same generators and setting $N_i=|\Omega_i|$, we obtain that
\begin{equation}\label{eq2}
|\mathcal{G}|=\sum_{T\in\Omega}\varphi(|T|)=\sum_{i=0}^{k}\sum_{T\in\Omega_i}\varphi(|T|)=\sum_{i=0}^{k}N_ip^{n+i-1}(p-1).
\end{equation}

Combining (\ref{eq1}) and (\ref{eq2}), we get $p^{n+k-1}=\sum_{i = 0}^{k}N_ip^{n+i-1} $. Hence there exists $j\in\{0,\ldots,k\}$ such that $q$ does not divide $N_{j}$. It follows that among the orbits of the action of $Q$ on $\Omega_{j}$ by conjugation there is a one-element orbit, in other words, there exists a group $T$ from $\Omega_{j}$ with $T^Q=T$. Since $T$ is generated by an element from $\mathcal{G}$, the required equality $HT=P$ holds.
\end{proof}
	
By the previous lemma, $R=TQ$ is a subgroup of $G$ and
$$R/(H\cap R)\simeq HR/H=HTQ/H=PQ/H=G/H\simeq D,$$
so $D$ is a quotient of~$R$. The group $R$ is a semidirect product of the cyclic $p$-subgroup $T$ and the subgroup $Q$ of order $q$, and, clearly, $[T,Q]\neq1$. Thus, applying Lemma~\ref{PropOfD} to $R$ instead of $D$, we conclude that normal subgroups of $ R $ form a chain. It follows that one of the subgroups $K_1\cap R$ and $K_2\cap R$ must be trivial, because they are normal and intersect trivially. This guarantees that $R$ can be embedded into either $X$ or $Y$, which, in turn, means that either $X$ or $Y$ contains a section isomorphic to~$D$. This completes the proof of the theorem.

\medskip

\textbf{Proof of the corollary.} We begin with the observation that the second statement of the corollary follows from the first one, because a variety generated by a finite number of finite groups is locally finite. Now, as $D$ is obviously finitely generated, we apply \cite[Lemma 4.3]{Higman} and conclude that $D$ is a section of a direct product of a finite number of groups from $\mathfrak{X}$. The required assertion now follows from the main theorem.

\medskip

\textbf{Last remarks}. The authors would like to thank Professor Evgeny Khukhro, who about 30 years ago drew the attention of one of them to  Kov\`acs' question. At the same time, Khukhro asked if the answer to the same question (at least in the case of finite groups $X$ and $Y$) is positive for an arbitrary finite group $D$, whose normal subgroups form a chain. We finish the paper with an example of such a group $D$ for which the answer to Kov\`acs' problem turns out to be negative.

Let $l\ge4$ be an even integer, $q$ an odd prime power, and $D$ the orthogonal group $\Omega_{2l}^+(q)$ (for the definitions and notations here and further see, e.g., \cite[\S~1.11]{Carter}). Since the quotient group $P\Omega_{2l}^+(q)$ of $D$ by its center of order $2$ is a simple group, the lattice of normal subgroups of $D$ is a chain. Consider now the spin group $G=\Spin_{2l}(q)$. The center $Z(G)$ of $G$ is an elementary abelian group of order~$4$. The quotient of $G$ by one of three subgroups of order $2$ from $Z(G)$, say $H$, is isomorphic to $D=\Omega_{2l}^+(q)$, while the quotients of $G$ by the other two subgroups, say $K_1$ and $K_2$, are isomorphic to the half spin group $\HS_{2l}(q)$ and are not isomorphic to~$D$. Now the Remak theorem yields that $G$ is a subdirect product of $G/K_1\times G/K_2\simeq HS_{2l}(q)\times HS_{2l}(q)$. Since $D=G/H$ is a section of $G$ and is not a section of $G/K_i$ for $i=1,2$, we get the desired example. We thank Alexander Buturlakin for the helpful hint to consider quasisimple groups.

\noindent{\sl Andrey V. Vasil'ev\\
Sobolev Insitute of Mathematics, 4 Acad. Koptyug avenue,\\
Novosibirsk State University, 2 Pirogova Str.,\\
Novosibirsk, 630090, Russia\\
e-mail: vasand@math.nsc.ru}

~

\noindent{\sl Saveliy V. Skresanov\\
Novosibirsk State University, 2 Pirogova Str.,\\
Novosibirsk, 630090, Russia\\
e-mail: s.skresanov@g.nsu.ru

\end{document}